\documentclass[12pt,a4paper]{article}

\usepackage{ amsmath,amsthm, amssymb, mathrsfs}

\newcommand{\R}{\mathbb{R}}

\newtheorem{theorem}{Theorem}[section]

\newtheorem{lemma}[theorem]{Lemma}

\def\be#1{\begin{equation}\label{#1}}
\def\ee{\end{equation}}

\def\e{{\rm e}}
\def\d{{\rm d}}

\title{Minimal periods of semilinear evolution equations with Lipschitz nonlinearity
  revisited}

\author{James C. Robinson\footnote{e-mail address:
    j.c.robinson@warwick.ac.uk}
  \qquad Alejandro Vidal-L\'opez\footnote{e-mail address: 
a.vidal-lopez@warwick.ac.uk}\\ \\
  {\small Mathematics Institute, University of Warwick, Coventry, CV4
    7AL. UK}}

\begin{document}
\maketitle

\begin{abstract}
  We show that when $A$ is a self-adjoint sectorial operator on a Hilbert space, for
  $0\le\alpha<1$ there exists a constant $K_\alpha$, depending only on $\alpha$, such that if $f:D(A^\alpha)\to X$ satisfies
$$
\|f(u)-f(v)\|_X\le L\|A^\alpha(u-v)\|_X
$$
then any periodic orbit of the equation $\dot u=-Au+f(u)$ has period at least
$K_\alpha L^{-1/(1-\alpha)}$. This generalises our previous result (J.\ Diff.\
Eq.\ 220 (2006), 396--406) which was restricted to $0\le\alpha\le 1/2$ and $A^{-1}$ compact.
\end{abstract}

\section{Introduction}

In 1969 Yorke proved a striking result providing lower bounds on the period of any periodic orbit of a Lipschitz ordinary differential equation. He showed that the period $T$ of any periodic orbit of
$$
\dot x=f(x),\ x\in\R^n,\qquad\mbox{with}\qquad|f(x)-f(y)|\le L|x-y|
$$
must satisfy $T\ge 2\pi/L$. This result was extended to ODEs on Hilbert spaces by Busenberg et al.\ (1986), who also proved the lower bound $T\ge 6/L$ in Banach spaces. For the results in this generality these two bounds are known to be sharp; for the Hilbert space bound one need only consider $(\dot x,\dot y)=(y,-x)$, while in the Banach space case an example is given by Busenberg et al.\ (1989). (Note, however, that the optimal bound in concrete Banach spaces, e.g. $(\R^n,\|\cdot\|_{\ell^p})$, is not known, despite some work in this direction, e.g.\ Zevin (2008), Nieuewenhuis \& Robinson (2012). Suggestively, Zevin (2012) has shown that if $Df(x)f(x)$ has Lipschitz constant $L$ in $(\R^n,\ell^p)$ then the period is $2\pi/L$ independent of $p$.).

In a previous paper (Robinson \& Vidal-L\'opez, 2006) inspired by work of Kukavica (1994) for the Navier--Stokes equations, we considered one natural analogue of this problem in the realm of partial differential equations, namely periodic orbits for semilinear evolution equations of the form
\be{eq:PDE}
\frac{\d u}{\d t}=-Au+f(u),
\ee
where $A$ was a positive self-adjoint operator with compact inverse, and $f$ was Lipschitz from $D(A^\alpha)$ into $H$ for some $0\le\alpha\le1/2$. In this case we showed that any periodic orbit must have period $T$ bounded below according to
\be{Tbound}
T\ge K_\alpha L^{-1/(1-\alpha)},
\ee
where $K_\alpha$ depends only on $\alpha$.

The bound in (\ref{Tbound}) gives no indication that the limitation of our analysis, namely $0\le\alpha\le1/2$, is in any way necessary. Indeed, we show in this paper that we can extend this to the full range $0\le\alpha<1$ and drop the requirement that $A$ has a compact inverse. This is the standard setting in which one can prove local existence and uniqueness results for (\ref{eq:PDE}), see Henry's 1981 monograph, for example.

In Section \ref{ODE} we recall the elegant proof of the bound $T\ge 2\pi/L$ in Hilbert spaces due to Busenberg et al.\ (1986). In Section \ref{PDE} when then prove the lower bound $T\ge K_\alpha L^{-1/(1-\alpha)}$ for (\ref{eq:PDE}). We indicate various applications in Section \ref{applications}.

\section{Lipschitz ordinary differential equations}\label{ODE}

First we give the very short and elegant proof due to Busenberg et al.\ (1986) of Yorke's lower bound on the period for Lipschitz differential equations in Hilbert spaces. The proof uses `Wirtinger's inequality', which is just the Poincar\'e inequality for functions defined on an interval.

\begin{lemma}\label{Wirtinger}
  Let $H$ be a Hilbert space with norm $\|\cdot\|$. If $f\in W^{1,2}(0,2\pi;H)$ and $\int_0^{2\pi}f(t)\,\d t=0$ then
  $$
  \int_0^{2\pi}\|f(t)\|^2\,\d t\le\int_0^{2\pi}\|\dot f(t)\|^2\,\d t.
  $$
\end{lemma}

\noindent The proof of the 1D inequality is straightforward using Fourier series. The inequality in $\R^n$ follows by applying the 1D inequality to each component of $f$, and the proof in a Hilbert space follows the same lines using an countable orthonormal set whose span contains $\cup_{t\in[0,2\pi]}f(t)$.

It is then easy to prove the following theorem.

\begin{theorem}[Yorke, 1969; Busenberg et al., 1986]
 Let $H$ be a Hilbert space. Any periodic orbit of the equation $\dot x=f(x)$, where $f:H\to H$ has Lipschitz constant $L$, has period $T\ge2\pi/L$.
\end{theorem}

\begin{proof}
  Since $x(\cdot)$ is periodic, for any $h>0$ the function $y(t)=x(t+h)-x(t)$ satisfies $\int_0^T y(t)\,\d t=0$. So we can use Lemma \ref{Wirtinger}:
  \begin{align*}
  \int_0^T \|x(t+h)-x(t)\|^2\,\mathrm{d}t&\le\left(\frac{T}{2\pi}\right)^2
  \int_0^T\|\dot x(t+h)-\dot x(t)\|^2\,\d t\\
  &=\left(\frac{T}{2\pi}\right)^2\int_0^T\|f(x(t+h))-f(x(t))\|^2\,\d t\\
  &\le\left(\frac{T}{2\pi}\right)^2\int_0^T L^2\|x(t+h)-x(t)\|^2\,\d t.
  \end{align*}
  It follows that $LT\ge2\pi$ as claimed.
\end{proof}

\section{Lipschitz semilinear evolution equations}\label{PDE}

Before we prove our main theorem we first recall some basic properties of self-adjoint operators on a
Hilbert space (e.g, see Reed \& Simon Vol I, 1980).

A projection-valued measure $\{P_\Omega\}_{\Omega\in
   \mathcal{B}(\R)}$ is a family of projections defined on the Borel sets of $\R$
 such that
 \begin{enumerate}
 \item each $P_\Omega$ is an orthogonal projection;
 \item $P_\emptyset = 0$, $P_{(-\infty,\infty)} = I$;
 \item if $\Omega = \cup_{n\geq 1} \Omega_n$ with $\Omega_n \cap\Omega_m =
   \emptyset$ if $m\not=n$, then
   \begin{displaymath}
     P_\Omega = \lim_{N\to \infty} \sum_{n=1}^N P_{\Omega_n}
   \end{displaymath}
   in the strong sense; and
 \item $P_{\Omega_1} P_{\Omega_2} = P_{\Omega_1 \cap \Omega_2}$.
 \end{enumerate}

 If $A$ is a self-adjoint sectorial operator, possibly unbounded, on a Hilbert
 space $H$, then by the spectral theorem (see Theorem VIII.6 in Reed \& Simon
 Vol I) there exists a
 projection-valued measure $\{P_\Omega\}_{\Omega\in \mathcal{B}(\R)}$ such that
 for any real-valued function $g(\lambda)$ defined on $\R$,
 \begin{displaymath}
   (\varphi, g(A)\varphi) = \int_{-\infty}^{+\infty} g(\lambda)
\, \mathrm{d}(\varphi, P_\lambda \varphi)
 \end{displaymath}
for any
 $\varphi \in D_g =\{\psi\in H:\ \int_\R |g(\lambda)|^2 \, \mathrm{d}(\psi,
 P_\lambda \psi)< \infty\}$.
From this it follows  by polarisation that for any $\varphi, \psi \in D_g$,
\begin{displaymath}
     (\varphi, g(A)\psi) = \int_{-\infty}^{+\infty} g(\lambda)
\, \mathrm{d}(\varphi, P_\lambda \psi),
\end{displaymath}
and in particular
 \begin{displaymath}
    (\varphi, A \varphi) = \int_{-\infty}^{+\infty} \lambda
\, \mathrm{d}(\varphi, P_\lambda \varphi)
 \end{displaymath}
 for any $\varphi \in D(A) = \{\psi\in H:\ \int_\R
 \lambda^2 \, \mathrm{d}(\psi, P_\lambda \psi)< \infty\}$.

As a consequence we can define projection operators as in the following lemma, which are the key ingredient in our proof. The existence of such projections is clear when $A^{-1}$ is compact and $H$ has a basis consisting of the eigenfunctions of $A$ (by choosing an $n$ such that $\lambda_n\le\mu<\lambda_{n+1}$ and letting $P$ be the projection onto the eigenfunctions corresponding to the first $n$ eigenvalues), which was the case we considered in our previous paper.

 \begin{lemma}
\label{lemma:spectralDecomp}
Let $A$ be a self-adjoint sectorial operator on $H$. Given $\mu>0$,
define projections $P = P_{[0,\mu)}$ and $Q=P_{[\mu, +\infty)}$. Then for the
operators $A_P = AP = PA$ and $A_Q = A Q = Q A$ defined in $PH$ and $QH$
respectively,
   \begin{displaymath}
    \|A_P\|_{PH} \leq \mu \quad \textrm{and} \quad
     \|(I-\mathrm{e}^{-A_Q T})^{-1}\|_{PH} < (1-\mathrm{e}^{-\mu T})^{-1}.
   \end{displaymath}
 \end{lemma}
 \begin{proof}
   First, notice that $P H$ and $QH$ are invariant subspaces for $A_P = AP =
   PA$ and $A_Q = A Q = Q A$.

   Taking $g(\lambda) = \lambda^2 \chi_{[0,\mu)}$ we have $D_g = P_{[0, \mu)}H$
   and so, for $\varphi\in P_{[0, \mu)}H$,
 \begin{displaymath}
   \|A_P \varphi\|^2 = (\varphi,A^2\varphi)=\int_{(-\infty,\mu)}\lambda^2
   \,\mathrm{d}(\varphi, P_\lambda \varphi)
   \leq \mu^2 \|\varphi\|^2,
 \end{displaymath}
i.e.\ $\|A_P\| \leq \mu$.

 Taking now $g(\lambda) = (1-\mathrm{e}^{-\lambda T})^{-1}\chi_{[\mu,+\infty)}$
 we have, for $\psi,\varphi\in D_g= P_{[\mu, +\infty)} H$ (we are using that $g(A_Q)$
 is a bounded operator on $P_{[\mu, +\infty)}H$),
 \begin{align*}
   (\psi, (I-\mathrm{e}^{-A_QT})^{-1}\varphi) &=
   \int_{[\mu,+\infty)}(1-\mathrm{e}^{-\lambda T})^{-1}
   \,\mathrm{d}(\psi, P_\lambda \varphi )\\
   & \leq (1-\mathrm{e}^{-\mu T})^{-1}
   \int_{[\mu,+\infty)}\,\mathrm{d}(\psi, P_\lambda \varphi ) .
 \end{align*}
 Therefore, taking the supremum in $\psi$ and then in $\varphi$, it follows that
 \begin{displaymath}
 \|(I-\mathrm{e}^{-A_Q T})^{-1}
\|_{H_2}\leq (1-\mathrm{e}^{-\mu T})^{-1}.
\end{displaymath}
 \end{proof}

We can now prove our  main theorem.

\begin{theorem}\label{Main}
  For each $\alpha$ with $0\le\alpha<1$ there exists a constant
  $K_\alpha$, depending only on $\alpha$, such that if $A$ is a
  self-adjoint sectorial operator on a Hilbert space with non-negative
  spectrum, then if
  $$
  \|f(u)-f(v)\|\le L\|A^\alpha(u-v)\|\qquad\mbox{for all}\qquad u,v\in D(A^\alpha)
  $$
  any periodic orbit of
   $$
   u_t+Au=f(u)
   $$
   must have period at least $K_\alpha L^{-1/(1-\alpha)}$.
\end{theorem}

Throughout the proof we use $\|\cdot\|$ for both the norm in $X$ and the operator norm in ${\mathscr L}(X,X)$.

\begin{proof}

  Suppose that (\ref{eq:PDE}) has a periodic orbit of minimal period
  $T>0$. Pick some $\tau$ with $0<\tau<T$, and let
  $D(t):=u(t)-u(t+\tau)$.

  Fix some $0<\delta<1/2$ and set $\mu = \delta/T$.   Since $A$ is self-adjoint, we can use Lemma \ref{lemma:spectralDecomp} to guarantee the existence of projections $P = P_{(-\infty,\mu)}$ and
  $Q=P_{[\mu, +\infty)}$ which are orthogonal to each other. In
  particular, $\|P\|\leq 1$ and $\|Q\|\leq 1$. Notice also that, since
  we are assuming that the spectrum of $A$ is contained in the
  nonnegative half-line, $P_\lambda = 0$ for all $\lambda <0$, that
  is, $P_\Omega =0$ for any $\Omega \subset (-\infty,0)$.
Moreover, $H = P H
 \oplus Q H$ whith $P H$ and $QH$ invariant for $A_P = AP =
 PA$ and $A_Q = A Q = Q A$. Furthermore, by Lemma \ref{lemma:spectralDecomp},
 \begin{displaymath}
   \|A_P\|_{PH} \leq \mu
\quad \mathrm{and} \quad
   \|(I-\mathrm{e}^{-A_Q T})^{-1}\|_{QH} \leq (1-\mathrm{e}^{-\mu
     T})^{-1}.
\end{displaymath}

First, write $p(t)=P D(t)$ and note following the proof of Theorem 2.1
in Robinson \& Vidal-L\'opez (2006) that
$$
p(t)-p(s)=\int_s^t\dot p(r)\,\d r.
$$
Integrating both sides with respect to $s$ from $0$ to $T$ gives
$$
Tp(t)=\int_0^T\left(\int_s^t\dot p(r)\,\d r\right)\,\d s
$$
and so
$$
 T\|A^\alpha p(t)\|\le\int_0^T\int_0^T\|A^\alpha\dot p(r)\|\,\d r\,\d s\le T\int_0^T\|A^\alpha\dot p(r)\|\,\d r.
 $$
Therefore
\begin{align*}
\|A^\alpha &PD(t)\|\le \int_0^T\|A^\alpha P[f(u(t+s))-f(u(t+\tau+s))]\|+\|A^\alpha PAD(t+s)\|\,\d s\\
&\le\int_0^T\|A^\alpha P\|\|f(u(t+s))-f(u(t+\tau+s))\|+\|AP\|\|A^\alpha PD(t+s)\|\,\d s\\
&\le\int_0^T \mu^\alpha\|f(u(s))-f(u(s-\tau))\|+\mu\|A^\alpha PD(s)\|\,\d s\\
&\le L\mu^\alpha\int_0^T\|A^\alpha D(s)\|\,\d s+\mu\int_0^T\|A^\alpha PD(s)\|\,\d s.
\end{align*}
To combine this with the $Q$ part we will need to raise everything to
the power of $q$ for some $q>1/(1-\alpha)$:
\begin{align*}
 \|A^\alpha PD(t)\|^q\le 2^qL^q\mu^{\alpha q}\left(\int_0^T\|A^\alpha D(s)\|\,\d s\right)^q+2^q\mu^q\left(\int_0^T\|A^\alpha PD(s)\|\,\d s\right)^q\\
\quad\le 2^qL^q\mu^{\alpha q}T^{q-1}\left(\int_0^T\|A^\alpha D(s)\|^q\,\d s\right)+2^q\mu^qT^{q-1}\int_0^T\|A^\alpha PD(s)\|^q\,\d s
\end{align*}
Integrating from $0$ to $T$ with respect to $t$ we obtain
\begin{align*}
\int_0^T\|A^\alpha &PD(t)\|^q\,\d t \le 2^q\mu^{\alpha q}T^qL^q\int_0^T\|A^\alpha D(s)\|^q\,\d s+2^q\mu^qT^q\int_0^T\|A^\alpha PD(s)\|^q\,\d s\\
&=2^q \delta^{\alpha q}T^{q(1-\alpha)}L^q\int_0^T\|A^\alpha
D(s)\|^q\,\d s+(2\delta)^q\int_0^T\|A^\alpha PD(s)\|^q\,\d s\\
&\le 2^{(1-2\alpha)q} \delta^{\alpha q} T^{q(1-\alpha)}L^q
\int_0^T\|A^\alpha D(s)\|^q\,\d s+(2\delta)^{q}\int_0^T\|A^\alpha PD(s)\|^q\,\d s,
\end{align*}
using the fact that $2\mu T=2\delta<1$ by our choice of $\mu$. Therefore
$$
\int_0^T\|A^\alpha PD(t)\|^q\,\d t\le \frac{2^{(1-2\alpha)q}}{1-(2\delta)^q}T^{q(1-\alpha)}L^q\int_0^T\|A^\alpha D(s)\|^q\,\d s.
$$
or \be{Pest2} \left(\int_0^T\|A^\alpha PD(t)\|^q\,\d t\right)^{1/q}\le
\frac{2^{(1-2\alpha)}}{(1-(2\delta)^q)^{1/q}} T^{1-\alpha}\left(\int_0^T\|A^\alpha
  D(s)\|^q\,\d s\right)^{1/q}.  \ee

For the $Q$ part, we note (following the first steps of the proof of Theorem 3.1 in Vidal-L\'opez \& Robinson, 2006)  that on a periodic orbit of period $T$
$$
u(t)=u(t+T)=\e^{-AT}u(t)+\int_0^T\e^{-A(T-s)}f(u(s+t))\,\d s,
$$
whence
$$
(I-\e^{-AT})u(t)=\int_0^T\e^{-A(T-s)}f(u(s+t))\,\d s.
$$
Therefore
\begin{align*}
D(t)&=(I-\e^{-AT})^{-1}\int_0^T\e^{-A(T-s)}[f(u(t+s))-f(u(t+\tau+s))]\,\d s\\
&=(I-\e^{-AT})^{-1}\int_0^T\e^{-A(T-s)}F(t+s)\,\d s,
\end{align*}
where $F(t):=f(u(t))-f(u(t+\tau))$, and so
\begin{equation}
  \label{eq:bounds:1}
  \|A^\alpha QD(t)\|\leq
\|(I - \mathrm{e}^{-AT})^{-1}Q\| \int_0^T\|A^\alpha\mathrm{e}^{-A(T-s)} \|
\|F(t+s)\|\,\mathrm{d}s.
\end{equation}

Now,  from Lemma 3.1
\begin{displaymath}
  \|( I - \mathrm{e}^{-A T})^{-1}\| \leq (1-\mathrm{e}^{-\mu T})^{-1} < \gamma,
\end{displaymath}
and since $\mu T = \delta <1/2$ it follows that $\gamma =
(1-\e^{-1/2})^{-1}\simeq 2.541$.

Using the inequality
$$
\|A^\alpha\e^{-At}\|\le M_\alpha t^{-\alpha}
$$
(where $M_\alpha=\alpha^\alpha\e^{-\alpha}$) we have
\begin{displaymath}
  \|A^\alpha QD(t)\|\leq \gamma \int_0^T M_\alpha (T-s)^{-\alpha}
L \|A^\alpha D(t+s)\|\,\mathrm{d}s.
\end{displaymath}

Now, we can apply H\"older's inequality to the integral term, with exponents $(p,q)$ where $\alpha p <1$, and obtain
\begin{displaymath}
   \|A^\alpha QD(t)\|\leq \gamma M_\alpha L\left(\int_0^T  (T-s)^{-\alpha p}
\,\mathrm{d}s\right)^{1/p}
\left(\int_0^T
\|A^\alpha D(t+s)\|^q\,\mathrm{d}s\right)^{1/q}
\end{displaymath}
Thus,
\begin{displaymath}
   \|A^\alpha QD(t)\|^q \leq
\left(\gamma M_\alpha L\right)^q
\left(\int_0^T  s^{-\alpha p}
\,\mathrm{d}s\right)^{q/p}
\int_0^T  \|A^\alpha D(t+s)\|^q\,\mathrm{d}s
\end{displaymath}
and so, noting that $q(1-\alpha p)/p=q(1-\alpha)-1$,
\begin{equation}\label{Qest1}
\|A^\alpha QD(t)\|^q\le \frac{(\gamma M_\alpha L)^q}{(1-\alpha p)^{q/p}}\,T^{q(1-\alpha)-1}\int_0^T\|A^\alpha D(s)\|^q\,\d s.
\end{equation}
Integrating from $0$ to $T$ we obtain
\begin{equation}\label{Qest2}
\int_0^T\|A^\alpha QD(t)\|^q\le\frac{(\gamma M_\alpha L)^q}{(1-\alpha p)^{q/p}}\,T^{q(1-\alpha)}\int_0^T\|A^\alpha D(s)\|^q\,\d s,
\end{equation}
or
\begin{equation}\label{Qest3}
\left(\int_0^T\|A^\alpha QD(t)\|^q\right)^{1/q}\le\frac{\gamma M_\alpha L}{(1-\alpha p)^{1/p}}\,T^{1-\alpha}\left(\int_0^T\|A^\alpha D(s)\|^q\,\d s\right)^{1/q}.
\end{equation}

Now combine (\ref{Pest2}) and (\ref{Qest3}) using the triangle inequality in $L^q$ to obtain
$$
\left(\int_0^T\|A^\alpha D(t)\|^q\,\d t\right)^{1/q}\le \left[\frac{2^{1-2\alpha}}{(1-(2\delta)^q)^{1/q}}+\frac{\gamma M_\alpha}{(1-\alpha p)^{1/p}}\right]LT^{1-\alpha}\left(\int_0^T\|A^\alpha D(s)\|^q\,\d s\right)^{1/q},
$$
which yields
$$
LT^{1-\alpha}\left[\frac{2^{1-2\alpha}}{(1-(2\delta)^q)^{1/q}}+\frac{\gamma M_\alpha}{(1-\alpha p)^{1/p}}\right]\ge 1.
$$

Finally letting $p\to1$ (and so $q\to\infty$) we obtain
$$
LT^{1-\alpha}\left[2^{1-2\alpha}+\frac{\gamma M_\alpha}{1-\alpha}\right]\ge 1,
$$
and so, recalling that $M_\alpha=\alpha^\alpha\e^{-\alpha}$,
$$
T\ge L^{-1/(1-\alpha)}\left[2^{1-2\alpha}+\frac{\gamma\alpha^\alpha\e^{-\alpha}}{1-\alpha}\right]^{-1/(1-\alpha)}.
$$
\end{proof}

Notice that the proof above only uses the fact that the spectrum of the operator
$A$ has a suitable decomposition, i.e.\ that given by Lemma
\ref{lemma:spectralDecomp}. In fact, the proof also works when $X$ is a Banach space and $A:D(A)\to X$ a sectorial operator for which there
exists a sequence of uniformly bounded projections $\{P_n\}_{n=1}^\infty$ that commute with all
powers of $A$, and an increasing sequence of positive real numbers
$\mu_n\to\infty$ such that
\be{split}
 \|AP_n\|_{{\mathscr  L}(X,X)}\le\mu_n
\qquad\mbox{and}\qquad\|AQ_n\|_{{\mathscr L}(X,X)}\ge\mu_n,
\ee
where $Q_n:=I-P_n$, and
\be{decay} \|A^\alpha\e^{-At}Q_n\|_{{\mathscr L}(X,X)}\le
M_\alpha t^{-\alpha}\e^{-\mu_n t},\qquad t\ge0,\ 0\le\alpha\le1.
\ee

\section{Some applications}\label{applications}

We now consider briefly some applications of this result.

\subsection{Reaction-diffusion equations}

While the system
$$
u_t-\Delta u=f(u,u_x)\qquad u|_{\partial\Omega}=0,\qquad\Omega\subset\R^n.
$$
is always gradient - and hence has no periodic orbits - if $f$ depends only on $u$, the introduction of dependence on $u_x$ means that this is no longer true. Note that if we consider the Nemytskii operator $F$ acting on functions $u$ and defined by
$$
F[u](x)=f(u(x),u_x(x)),
$$
smoothness properties of $F$ can be deduced from growth conditions on $f$. For
example, if
$$
|f(t,x) - f(s,y)| \leq C(1+ |t-s|^{p-1} + |x-y|^{q-1})(|t-s| + |x-y|)
$$
then $F:D(A^\alpha)\rightarrow H$, with $\alpha
=\max\left\{\frac{n(p-1)}{4},\frac{1}{2}+\frac{n(q-1)}{4}\right\}>\frac{1}{2}$. Thus
we can now treat this case, if $1<p< 1 + 4/n$ and $1<q<1+2/n$, that we were
unable to before.

We can also now consider the same problem set on the whole space,
\begin{displaymath}
  u_t - \Delta u = f(u,u_x), \quad \mathrm{in}\quad \R^n.
\end{displaymath}
We could not treat this before, since it is well-known that the spectrum of
$-\Delta$ is the half-line $[0,\infty)$ (e.g, see Reed \& Simon Vol 4, Example
XIII.4.6, p.117 for $L^2(\R^3)$ and Theorem XIII.15 (b) p. 119 for the general
case). However, such a system now falls within the framework of Theorem
\ref{Main}.

\subsection{Lotka--Volterra equations}

We take $\Omega\subset \R^N$ and consider the Lotka--Volterra system
$$  U_t - D \Delta U_t=F(U)\quad \mathrm{in}\ \Omega,\qquad
   \frac{\partial U}{\partial  n }= 0 \quad \mathrm{on}\ \partial \Omega,
$$
where $U = (u,v) ^\mathrm{T}$, $D = \mathrm{diag}(d_1,d_2)$, and
\begin{displaymath}
  F(U) =  \left(
    \begin{array}[]{l}
u(\lambda(x) - a(x) u - b(x) v)\\
 v(\mu(x) - c(x) u - d(x) v),
\end{array}
\right)
\end{displaymath}
with $a,b,c,d\in L^\infty(\Omega)$. The problem is well-posed in $\dot
H^\alpha(\Omega)\times \dot H^\alpha(\Omega)$ with $\alpha=N/2$ for $1\leq N\leq 3$, where $\dot
H^\alpha(\Omega)$ denotes the usual Sobolev space with elements having zero
mean. Notice that $F$ maps $\dot H^\alpha(\Omega)\times \dot H^\alpha(\Omega)$
into $L^2(\Omega) \times L^2(\Omega)$. Also, for $U,V\in B_R^\alpha$, the ball
of radius $R$ in $\dot H^\alpha(\Omega)$,
\begin{displaymath}
  \|F(U) - F(Y)\|^2_{L^2(\Omega)} \leq |L(R)|^2\|U-Y\|^2_{\dot H^\alpha(\Omega)}
\end{displaymath}
where
$$
|L(R)|^2=\max\{B_1(R),B_2(R)\}
$$
with
\begin{displaymath}
  B_1 (R)= 2 \left[\|\lambda\|^2_{\infty} +
C^4_\alpha R^2 (2\|a\|_\infty^2 + \|b\|_{\infty}^2 +
    \|c\|_{\infty}^2) \right]
\end{displaymath}
and
\begin{displaymath}
  B_2 (R)=2 \left[\|\mu\|^2_{\infty} +
C^4_\alpha R^2 (2\|d\|_\infty^2 + \|b\|_{\infty}^2 +
    \|c\|_{\infty}^2) \right],
\end{displaymath}
where $C_\alpha$ is the embedding constant in $\dot H^\alpha(\Omega)\subset
L^4(\Omega)$. Notice that $L(R)$ is increasing in $R$. It follows from
Theorem \ref{Main} that
\begin{displaymath}
  T^{1-\alpha} >\frac{1}{L(R)C_\alpha}
\end{displaymath}

Notice that, if the periodic solution $U(\cdot)$ is bounded in
$L^\infty(\Omega)$, we can consider the nonlinear term acting on the orbit as a
function $F: L^2 \to L^2$ which is Lipschitz on the periodic orbit with constant
$L=\max\{B_1,B_2\}$ with
\begin{displaymath}
  B_1= 2 \mathop{\mathrm{ess\ sup}\ }_{x\in\Omega}\{|\lambda(x)|^2 + M^2(2|a(x)|^2 + |b(x)|^2 +
  |c(x)|^2) \}
\end{displaymath}
and
\begin{displaymath}
  B_2 = 2 \mathop{\mathrm{ess\ sup}\ }_{x\in\Omega}\{|\mu(x)|^2 + M^2(2|d(x)|^2 + |b(x)|^2 +
  |c(x)|^2) \}
\end{displaymath}
where $M = \|U(\cdot)\|_{L^\infty(0,T;L^\infty(\Omega))}$. In particular, the
bound for the period in this case is given by
\begin{displaymath}
  T > c/L.
\end{displaymath}

\subsection{The 2D Navier--Stokes equations}

Finally we revisit the 2D incompressible Navier--Stokes equations,
$$
u_t-\Delta u+(u\cdot\nabla)u+\nabla p=f\qquad \nabla\cdot u=0,
$$
under periodic boundary conditions, which formed the main example in our previous paper.

Here we give a much simpler argument to obtain the same result, recalling that
$$
\|u\|_{L^p}\le cp^{1/2}\|Du\|\qquad\mbox{and}\qquad\|u\|_{L^\infty}\le c\epsilon^{-1}\|u\|_{H^{1+\epsilon}},
$$
(see Talenti (1976) and Bartuccelli \& Gibbon (2011), respectively).

We let $\Pi$ denote the orthogonal projector in $L^2$ onto divergence-free fields (`the Leray projector'), and define
$$
B(u,u)=\Pi[(u\cdot\nabla)u]\qquad\mbox{and}\qquad Au=-\Pi\Delta,
$$
enabling us to rewrite the governing equations as
$$
u_t+Au=B(u,u),
$$

Using the bilinearity of $B(u,u)$ we have
\begin{align*}
|B(u,u)-B(v,v)|&\le |B(u-v,u)|+|B(v,u-v)|\\
&\le \|u-v\|_{L^{2/\epsilon}}\|Du\|_{L^{2/(1-\epsilon)}}+\|v\|_{L^\infty}\|D(u-v)\|_{L^2}\\
&\le c(2/\epsilon)^{1/2}\|D(u-v)\|\|Du\|_{H^{1+*}}+c\epsilon^{-1/2}\|v\|_{H^{1+*}}\|D(u-v)\|_{L^2}\\
&=c\epsilon^{-1/2}\|D(u-v)\|[\|Du\|_{H^{1+\epsilon}}+\|Dv\|_{H^{1+\epsilon}}]\\
&\le c\epsilon^{-1/2}\|D(u-v)\|G^{1-\epsilon}G^{3\epsilon}.
\end{align*}
Minimising with respect to $\epsilon$ yields
$$
|B(u,u)-B(v,v)|\le cG(1+\log G)^{1/2}\|D(u-v)\|,
$$
and hence $T\ge cG^{-2}(1+\log G)^{-1}$ as before.

\end{document}